\documentclass[11pt, letterpaper]{article}    	
 \usepackage[margin=1in]{geometry}
\usepackage{setspace}
\usepackage{graphicx}	
\usepackage{subcaption}
\usepackage{float}
\usepackage[toc]{appendix}
\usepackage{lineno}
\usepackage{hyperref}

\usepackage[draft]{todonotes}   

\usepackage{enumerate}	
\usepackage{paralist}	

\usepackage{amssymb}
\usepackage{amsmath, amsthm, amssymb}

\newtheorem{thm}{Theorem}[section]
\newtheorem{conj}[thm]{Conjecture}
\newtheorem{prop}[thm]{Proposition}
\newtheorem{claim}{Claim}
\newtheorem{lemma}[thm]{Lemma}
\newtheorem{cor}[thm]{Corollary}

\theoremstyle{definition}
\newtheorem{defi}[thm]{Definition}
\def\eps{\varepsilon}

\usepackage{tikz}
\newcounter{casenum}

\newcommand*{\rom}[1]{\expandafter{\romannumeral #1\relax}}

\usepackage{authblk}
\makeatletter
\def\@cite#1#2{{\normalfont[{\bfseries#1\if@tempswa , #2\fi}]}}
\makeatother

\title{Integer colorings with forbidden rainbow sums
}
\author{Yangyang Cheng\thanks{School of Mathematics, Shandong University,
Jinan, China. Email: \texttt{mathsoul@mail.sdu.edu.cn}.}
\quad Yifan Jing\thanks{Mathematical Institute, University of Oxford, Oxfordshire, UK. Email: \texttt{yifan.jing@maths.ox.ac.uk}.}
\quad Lina Li\thanks{Department of Combinatorics \& Optimazation, University of Waterloo, Waterloo, Canada. Email: \texttt{lina.li@uwaterloo.ca}.}
\quad Guanghui Wang\thanks{School of Mathematics and Data Science Institute, Shandong University,
Jinan, China. Email: \texttt{ghwang@sdu.edu.cn}.}
\quad Wenling Zhou\thanks{School of Mathematics, Shandong University,
Jinan, China, and  Laboratoire Interdisciplinaire des Sciences du Num\'{e}rique, Universit\'{e} Paris-Saclay, Orsay, France. Email: \texttt{gracezhou@mail.sdu.edu.cn}.}
}
\date{}
\begin{document}

\maketitle

\begin{abstract}
For a set of positive integers $A \subseteq [n]$, 
an $r$-coloring of $A$ is rainbow sum-free if it contains no rainbow Schur triple. 
In this paper we initiate the study of the rainbow Erd\H{o}s-Rothchild problem in the context of sum-free sets, which asks for the subsets of $[n]$ with the maximum number of rainbow sum-free $r$-colorings. We show that for $r=3$, the interval $[n]$ is optimal, while for $r\geq8$, the set $[\lfloor n/2 \rfloor, n]$ is optimal.
We also prove a stability theorem for $r\geq4$. 
The proofs rely on the hypergraph
container method, and some ad-hoc stability analysis.
\end{abstract}

\bigskip


\section{Introduction}
An interesting direction of combinatorics in recent years is the study of multicolored versions of classical extremal results, whose origin can be traced back to a question of Erd\H{o}s and Rothschild~\cite{E} in 1974. They asked which $n$-vertex graph admits the maximum number of $2$-edge-colorings without monochromatic triangles, and conjectured that the complete balanced bipartite graph is the optimal graph. About twenty years later, Yuster~\cite{yuster1996number} confirmed this conjecture for {$n\ge 6$}.

\subsection{Erd\H{o}s-Rothschild problems in various settings}
There are many natural generalizations of the Erd\H{o}s-Rothschild problem. The most obvious one may be to ask it for graphs other than the triangles, and one may also increase the number of colors used. A graph G on $n$ vertices is called \textit{$(r, F)$-extremal} if it admits the maximum number of $r$-edge-colorings without any monochromatic copies of $F$ among all $n$-vertex graphs.
Alon, Balogh, Keevash and Sudakov~\cite{ABKS} 
showed that the Tur\'{a}n graph $T_{k}(n)$ is the unique $(r, K_{k+1})$-extremal graph for $k\geq 2$, $r\in\{2, 3\}$ and $n$ sufficiently large. Interestingly, they also showed that Tur\'{a}n graphs $T_{k}(n)$ are no longer optimal for $r\geq 4$. Indeed, Pikhurko, and Yilma~\cite{pikhurko2012maximum} later proved that $T_{4}(n)$ is the unique $(4, K_3)$-extremal graph, while $T_{9}(n)$ is the unique $(4, K_4)$-extremal graph. Determining the extremal configurations in general for $k\geq 2$ and $r\geq 4$ turned out to be a difficult problem. For further results along this line of research (when $F$ is a non-complete graph or a hypergraph), we refer to~\cite{hoppen2012edge, hoppen2014edge, hoppen2015edge, lefmann2013exact, lefmann2009colourings, lefmann2010structural}.

Another variant of this problem is to study edge-colorings of a graph avoiding a copy of $F$ with a prescribed color pattern. For an $r$-colored graph $\hat{F}$, a graph $G$ on $n$ vertices is called $(r, \hat{F})$-extremal if it admits the maximum number of $r$-colorings which contain no subgraph whose color pattern is isomorphic to $\hat{F}$.
This line of work was initiated by Balogh~\cite{balogh2006remark}, who showed that the Tur\'{a}n graph $T_{k}(n)$ once again yields the maximum number of 2-colorings avoiding $H_{k+1}$, where $H_{k+1}$ is any 2-coloring of $K_{k+1}$ that uses both colors. For $r\geq 3$, the behavior of $(r, H_{k+1})$-extremal  graphs was studied by Benevides, Hoppen, Sampaio, Lefmann, and Odermann, see~\cite{BHS, hoppen2015color, hoppen2017graphs, hoppen2017rainbow}. In particular, the case when $\hat{F}=\hat{K}_3$ is a triangle with rainbow pattern has recently received a lot of attention (for its relation to Gallai colorings).
Hoppen, Lefmann and Odermann~\cite{hoppen2017graphs} first proved that the Tur\'{a}n graph $T_2(n)$ is the unique $(r, \hat{K}_3)$-extremal graph for $r\geq 5$. Very recently, Balogh and Li~\cite{balogh2018typical}, confirming conjectures of~\cite{BHS} and~\cite{hoppen2017graphs}, showed that the complete graph $K_{n}$ is the unique $(3, \hat{K}_3)$-extremal graph, while the Tur\'{a}n graph $T_2(n)$ becomes optimal as $r\geq 4$.

The Erd\H{o}s-Rothschild problem can also be extended to other discrete structures. In the domain of extremal set theory, Hoppen, Kohayakawa and Lefmann~\cite{hoppen2012hypergraphs} solved the Erd\H{o}s-Rothschild extension of the famous Erd\H{o}s-Ko-Rado Theorem. They, for instance, showed that the optimal $\ell$-intersecting families (each set is of size $k$) yields the maximum number of $r$-colorings in which every color class is $\ell$-intersecting for $r\in\{2, 3\}$, and also provided a fairly complete characterization of the corresponding extremal family for $r\geq 4$. Hoppen, Lefmann and Odermann~\cite{hoppen2016coloring}, and Clemens, Das and Tran~\cite{clemens2018colourings} later studied the Erd\H{o}s-Rothschild extension of the Erd\H{o}s-Ko-Rado Theorem for vector spaces. Moving the problem to the context of power set lattice, recently, Das, Glebov, Sudakov and Tran~\cite{das2019colouring} investigated the Erd\H{o}s-Rothschild extension of Sperner's Theorem, and proved that the largest antichain yields the maximum number of $r$-colorings, in which each color class is an antichain, for $r\in\{2, 3\}$. As for many of the previous results, they demonstrated that as $r$ grows, the largest antichain is no longer optimal. They also determined that the extremal configurations for 2-colorings without monochromatic $k$-chains are the largest $k$-chain-free family. The extremal configurations for $r\geq 3$ and $k\geq 2$ are widely unknown.
It came recently to our attention that
the rainbow Erd\H{o}s-Rothschild problem has 
also been extended to other discrete structures. Li, Broersma and Wang~\cite{2022Integer} studied the number of $r$-colorings of $[n]$
without rainbow 3-term arithmetic progressions and Lin, Wang and the fifth author~\cite{lin2022integer}
considered the problem for $k$-term arithmetic progressions.

\subsection{Erd\H{o}s-Rothschild problems for sum-free sets}
Given integers $n\geq  m \geq 1$, write $[m,n]:=\{m,\ldots,n\}$ and $[n]:=\{1,...,n\}$.

\begin{defi}[Schur triple \& Sum-free set]
A \textit{Schur triple} or a \textit{sum} in an abelian group $G$ (or in $[n]$) is a triple $(a, b, c)$ with $a+b=c$. A set $A\subseteq G$ (or $A\subseteq [n]$) is \textit{sum-free} if $A$ contains no such triple. 
\end{defi}
Given a set A of numbers,  an \textit{r-coloring} of $A$ is a mapping $f: A \rightarrow [r]$, which assigns one color to each element of $A$. An $r$-coloring of A is called a \textit{sum-free $r$-coloring} if each of the color classes is a sum-free set. Sum-free colorings are among the classical objects studied in extremal combinatorics and can be traced back to Schur's theorem, one of the seminal results in Ramsey theory.

The Erd\H{o}s-Rothschild extension for sum-free sets has been pursued by Liu, Sharifzadeh and Staden~\cite{liu2017maximum} for subsets of the integers, and H\`{a}n and Jim{\'e}nez~\cite{han2018maximum} for finite abelian groups. More specifically, they investigated the extremal configurations which maximize the number of sum-free $r$-colorings.
In the setting of integers, it is well known that the largest sum-free set in $[n]$ has size $\lceil n/2\rceil$. Liu, Sharifzadeh and Staden~\cite{liu2017maximum} determined the extremal configurations for $r=2$.

\begin{thm}{\rm (\cite[Theorem 1.2]{liu2017maximum}).} 
There exists $n_0 > 0$ such that for all integers $n\geq n_0$, the number of sum-free $2$-colorings of a subset $A\subseteq [n]$ is at most $2^{\lceil n/2\rceil}$. Moreover, the extremal subsets are $\{1, 3, 5, \cdots, 2\lceil n/2\rceil-1\}$, and $[\lfloor n/2\rfloor+1, n]$; and if $n$ is even, we additionally have $[n/2, n-1]$, and $[n/2, n]$.
\end{thm}

Unlike the graph case, in the sum-free setting, there are extremal configurations which are not sum-free even for 2 colors. Therefore, one would expect a more sophisticated extremal behavior as $r$ grows. Although some asymptotic bounds were obtained in~\cite{liu2017maximum}, the characterization of extremal sets for $r\geq 3$ remains widely open.

Such problem was also studied for finite abelian groups. Let $G$ denote a finite abelian group. Over fifty years ago, Diananda and Yap~\cite{diananda1969maximal} determined the maximum density $\mu(G)$ of a sum-free set in $G$ whenever $|G|$ has a prime factor $q \not\equiv 1 \mod 3$, but it was not until 2005 that Green and Ruzsa~\cite{green2005sum} completely solved this extremal question for all finite abelian group. H\`{a}n and Jim{\'e}nez~\cite{han2018maximum} investigated the Erd\H{o}s-Rothschild extension for sum-free sets on some special abelian groups.

\begin{thm}{\rm (\cite[Theorem 3]{han2018maximum}).}
Let $r\in \{2,3\}$, $q\in \mathbb{N}$ and let $G$ be a abelian group of sufficiently large order, which has a prime divisor $q$ such that $q\equiv 2 \mod 3$. Then the number of sum-free $r$-colorings of a set $A\subseteq G$ is at most $r^{\mu(G)}$. Moreover, the maximum is only achieved by the largest sum-free set.
\end{thm}
For more than three colors this phenomenon does not persist in general and the problem becomes considerably more complicated. For more details, we refer the readers to~\cite{han2018maximum}. For other abelian groups, despite some asymptotic bounds presented in~\cite{han2018maximum}, the exact extremal phenomena is unknown even for 2 colors.

\subsection{Our results}
In this paper, we consider a rainbow variant of the Erd\H{o}s-Rothschild problem for sum-free sets in $[n]$.
A Schur triple or a sum $(x,y,z)$ is a \textit{rainbow sum} if $x,y,z$ are colored with different colors. Note that a rainbow sum must have three distinct elements. For convenience, sometimes we would use the following definitions, which are slightly different to the classical notations on sum-free sets.
\begin{defi}[Restricted Schur triple \& Restricted sum-free set]
A \textit{restricted Schur triple} or a \textit{restricted sum} in $[n]$ is 
a triple $\{a, b, c\}$ with $a<b<c$ and $a+b=c$. A set $A\subseteq [n]$ is \textit{restricted sum-free} if $A$ contains no such triple.
\end{defi}

For any integer $n\geq7$, it is not hard to show that the largest restricted sum-free sets in $[n]$ have size $\lfloor n/2 \rfloor+1$. If $n$ is even, then the only subset attaining this bound is $\left[\frac{n}{2}, n\right]$; if $n$ is odd, then the maximum restricted sum-free sets are attained by the following four sets: $\left\{\frac{n-1}{2}, \frac{n-1}{2}+1, \ldots, n-1\right\}$, $\left\{\frac{n-1}{2}, \frac{n-1}{2}+2, \ldots, n\right\}$, $\left[\frac{n+1}{2}, n\right]$, and $\{1, 3, 5, \ldots, n\}$. 

Given a set of positive integers $A \subseteq [n]$, an $r$-coloring of $A$ is \textit{rainbow sum-free} if it contains no rainbow sum.
For a positive integer $r$ and a set $A\subseteq [n]$, we write $g(A, r)$ for the number of rainbow sum-free $r$-colorings of $A$ and define
\[
g(n, r):=\max_{A\subseteq [n]}g(A,r).
\]
A set $A\subseteq [n]$ is \textit{rainbow $r$-extremal} if $g(A, r)=g(n, r)$.
When $r\in\{1,2\}$, it is trivial to see that $g(n, r)=r^n$ for all positive integers $n$, and the only extremal set is the interval $[n]$, since for every subset $A\subseteq [n]$, all $r$-colorings of $A$ are rainbow sum-free. For $r \geq 3$, the characterization of the extremal sets requires substantially more work.

Our first main result is an upper bound on the number of rainbow sum-free $r$-colorings of dense sets.
\begin{thm}\label{thm: hrange}
For every integer $r\geq 3$, there exists $n_0\in \mathbb{N}$ such that for all $n>n_0$ the following holds. For a set $A\subseteq [n]$ with $|A|\geq (1-{r}^{-3})n$, the number of rainbow sum-free $r$-colorings $g(A,r)$ satisfies
$$g(A, r)\leq \binom{r}{2}\cdot 2^{|A|}+2^{-\frac{n}{26\log n}}2^n.$$
\end{thm}

By choosing two of the $r$ colors and coloring the elements of $[n]$ arbitrarily with these two colors, one can easily obtain that
\begin{equation}\label{eq: int}
g([n], r)\geq \binom{r}{2}(2^n - 2) + r =\binom{r}{2}2^n - (r^2 - 2r).
\end{equation}
Therefore, Theorem~\ref{thm: hrange} is asymptotically sharp for $A=[n]$ and then the typical structure of rainbow sum-free $r$-colorings of $[n]$ immediately follows from~(\ref{eq: int}).
\begin{cor}\label{cor1}
For every integer $r\geq 3$, almost all rainbow sum-free $r$-colorings of $[n]$ are 2-colorings.
\end{cor}

Now we turn to the extremal configurations of rainbow sum-free $r$-colorings. Let us first consider the case $r=3$. Similarly as in the Gallai coloring problem, two natural candidates of the extremal sets are the maximum restricted sum-free sets and the interval $[n]$.
Note that for every restricted sum-free set $A$,  we have $g(A, 3)\le 3^{\lfloor n/2 \rfloor+1} < g([n], 3)$ for all $n\ge 3$. Our second theorem shows that for three colors the interval $[n]$ is indeed optimal.

\begin{thm}\label{thm: three}
There exists $n_0\in \mathbb{N}$ such that for all $n > n_0$, among all subsets of $[n]$, the interval $[n]$ is the unique rainbow 3-extremal set.
\end{thm}

Just as for the Erd\H{o}s-Rothschild extension for Gallai colorings~\cite{balogh2018typical}, we may not expect that the same phenomena persist for $r\geq 4$.
For $n\in \mathbb{N}$, we define 
\[
O:=\{1, 3, 5, \cdots, 2\lceil n/2\rceil-1\}, \quad\text{ and }\quad  I_0:=[\lfloor n/2\rfloor+1, n].
\]
We prove the following stability theorem.
\begin{thm}\label{thm:sta}
For every positive integer $r\geq 4$, we have 
\[
g(n, r)= r^{n/2 + o(n)}.
\]
Moreover, for every $\eps>0$, there exist $\delta , n_0>0$ such that for all integers $n\geq n_0$ the following holds.
Let $A$ be a subset of $[n]$ with $g(A, r)\geq  r^{n/2 - \delta n}.$ Then
\smallskip
\begin{compactenum}
\item[\rm (i)] for $r\geq 5$, we have that either $|A \bigtriangleup O|\leq \eps n$, or $|A	\bigtriangleup I_0|\leq \eps n$;\smallskip
\item[\rm (ii)] for $r=4$, we have that either $|A	\bigtriangleup [n]|\leq \eps n$, or $|A 	\bigtriangleup O|\leq \eps n$, or $|A\bigtriangleup I_0|\leq \eps n$.
\end{compactenum}
\end{thm}

The behavior of the exact extremal configurations not only depends on the number of colors, but also depends on the parity of $n$.
For even $n$, we define
\[
I_1=\left[\frac{n}{2}-1, n\right], \qquad
I_2=\left[\frac{n}{2}, n\right].
\]
Observe that $I_1$ contains exactly two restricted Schur triples $\{n/2-1, n/2, n-1\}$, $\{n/2-1, n/2+ 1, n\}$, and it is not hard to compute that $g(I_1, r)=r^{n/2}\left(3 - 2/r\right)^2$. On the other hand, the set $I_2$ is a restricted sum-free set and therefore $g(I_2, r)=r^{|I_2|}=r^{n/2 + 1}$.
For odd $n$, we define
\[
I_3=\left[\frac{n-1}{2}, n\right].
\]
Again, the set $I_3$ contains exactly one restricted Schur triple $\{\frac{n-1}{2}, \frac{n-1}{2}+1, n\}$, and one can show that $g(I_3, r)=r^{\lceil n/2\rceil}\left(3 - 2/r\right)$, which is already greater than the number of colorings for any restricted sum-free set.
When a set $A$ is of size at least the size of the maximum restricted sum-free sets and not one of the above three sets, we believe that the restrictions from the triples would more than counteract the extra possibilities offered by the additional vertices. Therefore, we make the following conjecture.
\begin{conj}\label{conj}
Let $n, r$ be positive integers  and $r\geq 4$.\smallskip
\begin{compactenum}[\rm (i)]
\item If $n$ is even and $r\leq 7$, then $g(n, r)=r^{n/2}\left(3 - 2/r\right)^2$, and $I_1$ is the unique rainbow $r$-extremal set.\smallskip
\item If $n$ is even and $r\geq 8$, then $g(n, r)=r^{n/2 + 1}$, and $I_2$ is the unique rainbow $r$-extremal set.\smallskip
\item If $n$ is odd and $r=4$, then $g(n, r)=g([n], r)$, and $[n]$ is the unique rainbow $r$-extremal set.\smallskip
\item If $n$ is odd and $r\geq 5$, then $g(n, r)=r^{\lceil n/2\rceil}\left(3 - 2/r\right)$, and $I_3$ is the unique rainbow $r$-extremal set.
\end{compactenum}
\end{conj}
Our fourth main result verifies Conjecture~\ref{conj} for $r\geq 8$ and $n$ sufficiently large.
\begin{thm}\label{thm: r8}
For an integer $r\geq 8$, there exists $n_0=n_0(r)$ such that for all $n > n_0$ the following holds. Let $A$ be a subset of $[n]$ with $|A|\geq\lceil n/2 \rceil +1$. \smallskip
\begin{compactenum}[\rm (i)]
\item If $n$ is even, then $g(A, r)\leq r^{\lceil n/2 \rceil +1}$, and the equality holds if and only if $A=I_2$.\smallskip
\item If $n$ is odd, then $g(A, r)\leq r^{\lceil n/2\rceil}\left(3 - 2/r\right)$, and the equality holds if and only if $A=I_3$.
\end{compactenum}
\end{thm}

The paper is organized as follows. In the next section, we list some structural results on sum-free sets, which are essential for the proof, and introduce the multi-color container theorem. In Section~3, we prove Theorem~\ref{thm: hrange}. In Section 4, we prove the stability theorem, Theorem~\ref{thm:sta}, and determine $g(n,3)$ for $n$ sufficiently large. In Section 5, we determine $g(n,r)$ for $r\geq8$, and describe the corresponding extremal configurations. 
We close the paper with some concluding remarks in Section 6.
Throughout the paper, all logarithms have base 2. 

\section{Notation and preliminaries}
\subsection{Basic properties of restricted sum-free sets}
We use the following result of Huczynska~\cite{sum-counting} on the minimum number of additive triples
among all sets of a given size. 
 \begin{thm}\cite{sum-counting}\label{thm:staden}
Let $A$ be a subset of $[n]$ with $|A|=k > \lceil n/2 \rceil$. Then the number of Schur triples in $A$ is at least
 \[
 (k - \lceil n/2 \rceil)(k - \lfloor n/2 \rfloor),
 \]
and the unique set for any such given $k$ that attains this bound is $[n - k + 1, n]$.
 \end{thm}
  For a set $A \subseteq [n]$, we write $\mathcal{S}(A)$ for the set of all restricted Schur triples in $A$, and let $s(A)=|\mathcal{S}(A)|$. For an integer $t\in A$, denote by $\mathcal{S}(t, A)$ the set of all triples in $\mathcal{S}(A)$ containing $t$, and let $s(t, A)=|\mathcal{S}(t, A)|$. Then from Theorem~\ref{thm:staden}, it is not hard to  obtain the following proposition.

\begin{prop}\label{numberofsum}
For every integer $n\ge 3$, the number of restricted 
 Schur triples in $[n]$ satisfies that \[
s([n])= \begin{cases}
\frac{n^2-2n}{4}&\text{if } n \text{ is even};\\
\frac{n^2-2n+1}{4}&\text{otherwise}.
\end{cases}
\]
\end{prop}

For convenience, we will use the following definitions in sections 4 and 5.  For a subset $A \subseteq [n]$ and $t\in A$, let the \textit{link graph} $L_t(A)$ be the simple graph defined on the vertex set $A\setminus\{t\}$, such that $xy\in E(L_t(A))$ if and only if $\{t, x, y\}\in \mathcal{S}(t, A)$.  
Let $k(t, A)$ be the size of a maximum matching of $L_t(A)$. Observe that $\Delta(L_t(A))\leq2$, and $|E(L_t(A))|=s(t,A)$.
Therefore we have 
\begin{equation}\label{ineq:mat}
k(t, A)\geq |E(L_t(A))|/2=s(t, A)/2.
\end{equation}


%
%

\subsection{Structural properties of sum-free sets}

We will use standard definitions and notation in additive combinatorics as given in \cite{additive}. Given $A,B\subseteq \mathbb{Z}$, let
\[
A+B:=\{a+b: a\in A,b\in B\},\quad\text{ and }\quad A-B:=\{a-b: a\in A,b\in B\}.
\]
When $B=\{x\}$, we simply write $A+x$ and $A-x$. 

The following lemma is known as Green's removal lemma, which was first proved by Green \cite{green2005re-lemma}, and was later generalized to non-abelian groups by Kr\'al and Vena \cite{KSV2009}. 
\begin{lemma}{\rm (\cite[Corollary 1.6]{green2005re-lemma}).}\label{Stabilitylem2}
Suppose that $A \subseteq [n]$ is a set containing $o(n^2)$ Schur triples. Then there exist $B, C\subseteq [n]$ such that  $A=B\cup C$ where $B$ is sum-free and $|C|=o(n)$.
\end{lemma}

Note that the notation here is convenient but offers scope for confusion. What we mean is that there is a function $\delta=\delta(\eps)$ such that $\delta \to 0$ as $\eps \to 0$, and which has the following property. If $A$ contains at most $\delta n^2$ Schur triples then we may remove $\varepsilon n$ elements from $A$ so as to leave a set which is sum-free.

We also require a very strong stability theorem for sum-free sets proved by Deshouillers, Freiman, S\'os, and Temkin \cite{Deshouillers1999sum-free}.

\begin{lemma}{\rm (\cite[Theorem 1.1]{Deshouillers1999sum-free}).} \label{Stabilitylem3}
Every sum-free set $S$ in $[n]$ satisfies at least one of the following conditions:\smallskip
\begin{compactenum}[\rm (i)]
  \item $|S|\leq \lceil 2n/5 \rceil$;
  \item $S$ consists of odd numbers;
  \item $|S| \leq \min(S)$.
\end{compactenum}
\end{lemma}

\subsection{Multi-color container theorem}
An important tool in our proof is the hypergraph container theorem.
We use the following version from~\cite{balogh2017container}. Let $\mathcal{H}$ be a $k$-uniform hypergraph with average degree \textit{d}. The \textit{co-degree} of a set of vertices $X\subseteq V(\mathcal{H})$ is the number of edges containing $X$, that is,
\[
d(X)=|\{e\in E(\mathcal{H})\ |\ X\subseteq e\}|.
\]
For every integer $2 \leq j\leq k$, the $j$-th maximum co-degree of $\mathcal{H}$ is
\[
\Delta_j(\mathcal{H})=\max\{d(X)\ |\ X\subseteq V(\mathcal{H}),\ |X|=j\}.
\]
When the underlying hypergraph is clear, we simply write it as $\Delta_j$. For $0<\tau<1$, the \textit{co-degree function} $\Delta(\mathcal{H},\tau)$ is defined as
\[
\Delta(\mathcal{H},\tau)=2^{\binom{k}{2}-1}\sum_{j=2}^k2^{-\binom{j-1}{2}}\frac {\Delta_j}{d\tau^{j-1}}.
\]
In particular, when $k=3$,
\[
\Delta(\mathcal{H},\tau)=\frac {4\Delta_2}{d\tau}+\frac {2\Delta_3}{d\tau^{2}}.
\]

\begin{thm}{\rm (\cite[Theorem 3.1]{balogh2017container}).} \label{baloghcontainer}
Let $\mathcal{H}$ be a $k$-uniform hypergraph on vertex set $[N]$. Let $0<\varepsilon, \tau< 1/2$.
Suppose that $\tau<1/(200k!^2k)$ and $\Delta(\mathcal{H},\tau)\leq \varepsilon/(12k!)$. Then there exists $c=c(k)\leq 1000k!^3k$ and a collection of vertex subsets $\mathcal{C'}$ such that\smallskip
\begin{compactenum}[\rm (i)]
  \item every independent set in $\mathcal{H}$ is a subset of some $U\in \mathcal{C'}$;\smallskip
 \item for every $U\in \mathcal{C'}$, $|E(\mathcal{H}[U])|\leq \varepsilon\cdot |E(\mathcal{H})|$;\smallskip
  \item $\log|\mathcal{C'}|\leq cN\tau \log(1/\varepsilon)\log(1/\tau)$.
\end{compactenum}
\end{thm}

 A key concept in applying container theory to such coloring problems is the notion of \textit{template}, which was first introduced in \cite{FOSU}, although the concept had already appeared in~\cite{ST} under the name of `2-colored multigraphs' and later in~\cite{balogh2016further}, simply referred to as `containers'.
\begin{defi}[Template\ and\ palette]
An \textit{$r$-template} of order $n$ is a function $P: [n] \to 2^{[r]}$, associating with each element $x \in [n] $ a list of colors $P(x) \subseteq [r]$. We refer to this set $P(x)$ as the \textit{palette} available at $x$.

For a set $A \subseteq [n]$, any $r$-coloring of $A$ can be considered as an $r$-template of order $n$, with only one color allowed at each element in $A$, and no color allowed for elements not belonging to $A$.
\end{defi}

\begin{defi}[Subtemplate]
Let $P_1$, $P_2$ be two $r$-templates of order $n$. We say that $P_1$ is a \textit{subtemplate} of $P_2$ (written as $P_1 \subseteq P_2$) if $P_1(x) \subseteq P_2(x)$  for each element $x \in [n]$.
\end{defi}

For an $r$-template $P$ of order $n$, we say that $P$ is a \textit{rainbow restricted sum $r$-template} if there exists a restricted sum $S=\{a,b,c\}$ in $[n]$ such that $|P(x)|=1$ for $x\in S$, $P(x)=\varnothing$ for  $x\in [n]\setminus S$, and $P(a), P(b), P(c)$ are pairwise distinct.
Moreover, we write $RS(P)$ for the number of subtemplates of $P$ that are rainbow restricted sum. We say that $P$ is a \textit{rainbow restricted sum-free} $r$-template if $RS(P)=0$.

Using Theorem \ref{baloghcontainer}, we obtain the following.
\begin{thm} \label{container}
For every integer $r\geq 3$, there exists a constant $c=c(r)$ and a collection $\mathcal{C}$ of $r$-templates of order $n$ such that\smallskip
\begin{compactenum}[\rm (i)]
  \item every rainbow restricted sum-free $r$-template of order $n$ is a subtemplate of some $P\in \mathcal{C}$;\smallskip
  \item for every $P\in \mathcal{C}$, $RS(P)\leq n^{-1/3}s([n])$;\smallskip
  \item $|\mathcal{C}|\leq 2^{cn^{2/3} \log^2n}$.
\end{compactenum}
\end{thm}

\begin{proof}
Let $\mathcal{H}$ be a $3$-uniform hypergraph with vertex set $[n]\times \{1,2,...,r\}$, whose edges are all triples $\{(x_1, c_1), (x_2, c_2), (x_3, c_3)\}$ such that $(x_1, x_2, x_3)$ forms a restricted Schur triple in $[n]$ and $c_1, c_2, c_3$ are all different. In other words, every hyperedge in $\mathcal{H}$ corresponds to a rainbow restricted Schur triple. 
Note that there are exactly $r(r-1)(r-2)$ ways to rainbow color a restricted Schur triple with $r$ colors.  By Proposition~\ref{numberofsum}, the average degree $d$ of $\mathcal{H}$ is equal to
\[
d=\frac{3|E(\mathcal{H})|}{|V(\mathcal{H})|}
=\frac{3r(r-1)(r-2)s([n])}{nr}
\geq \frac{3(r-1)(r-2)n}{8}.
\]
Now we apply Theorem \ref{baloghcontainer} to $\mathcal{H}$. Let $\varepsilon=n^{-1/3}/r(r-1)(r-2)$ and
$\tau=\sqrt{96\cdot 3! \cdot r}n^{-\frac{1}{3}}$.
Observe that $\Delta_2(\mathcal{H})=2(r-2), \Delta_3(\mathcal{H})=1.$
For $n$ sufficiently large, we can get $\tau<1/(200\cdot3!^2\cdot3)$ and 
\[
\Delta(\mathcal{H},\tau)=\frac{4\Delta_2}{d\tau}+\frac{2\Delta_3}{d\tau^2}=\frac{8(r-2)}{d\tau}+\frac{2}{d\tau^2}\leq \frac{3}{d\tau^2}\leq \frac{\varepsilon}{12\cdot 3!}.
\]
Hence, there is a collection of vertex subsets $\mathcal{C'}$ satisfying properties (i)-(iii) of Theorem \ref{baloghcontainer}. Observe
that every vertex subset of $V(\mathcal{H})$ corresponds to an $r$-template of order $n$; every rainbow restricted sum-free $r$-template of order $n$ corresponds to an independent set in $\mathcal{H}$. Therefore, $\mathcal{C'}$ corresponds to a collection $\mathcal{C}$ of $r$-templates of order $n$ which satisfies properties (i)-(iii) of Theorem~\ref{container}.
\end{proof}

\begin{defi}[Good $r$-template]
For $A\subseteq [n]$, an $r$-template $P$ of order $n$ is a \textit{good $r$-template} of $A$ if it satisfies the following properties:\smallskip
\begin{compactenum}[(i)]
  \item For each element $i \in A$, $|P(i)|\geq 1$;\smallskip
  \item $RS(P)\leq n^{-1/3}s([n])$.
\end{compactenum}
\end{defi}

For a set $A\subseteq [n]$ and a collection of templates $\mathcal{P}$, denote by $G(\mathcal{P},A)$ the set of subtemplates $P'$ of some $P\in \mathcal{P}$ such that $P'$ is a rainbow sum-free $r$-coloring of $A$. 
Let $g(\mathcal{P},A)=|G(\mathcal{P},A)|$.
If $\mathcal{P}$ consists of a single $r$-template $P$, then we simply write $G(P,A)$ and $g(P,A)$.

\section{Proof of Theorem \ref{thm: hrange}}

Throughout this section, we fix an integer $r\geq 3$, a sufficiently large integer $n$ and an arbitrary set $A\subseteq [n]$ with $|A|=(1-\xi)n$, where 
\[
0\leq \xi \leq r^{-3}.
\] 
Let $\mathcal{C}$ be the collection of containers given by Theorem \ref{container}, and $\delta =1/(24\log n)$. 
We divide $\mathcal{C}$ into two classes
\begin{equation}\label{def:coll}
\mathcal{C}_1=\{P\in \mathcal{C} : g(P,A)\leq 2^{(1-\delta)n}\},
\quad
\mathcal{C}_2=\{P\in \mathcal{C} : g(P,A)> 2^{(1-\delta)n}\}.
\end{equation} 
Note that every $P\in\mathcal{C}_2$ is a good $r$-template of $A$.
The crucial part of the proof is to estimate $g(\mathcal{C}_2,A)$, which relies on the following four lemmas.

\begin{lemma}\label{(x,y)sumfree}
Let $F$ be the collection of pairs $\{a,b\}\subseteq  A$ with $a<b$ such that $\{a,b\}$ is not contained in any restricted Schur triple of $A$.
Then we have $|F|\leq \xi n^2+n/6$.
\end{lemma}
\begin{proof}
Let
\begin{gather*}
F_1= \{\{a,b\}\subseteq  A: a+b\in [n]\setminus  A, \ b=2a\},
\quad
F_2= \{\{a,b\}\subseteq  A: a+b>n, \ b=2a\},\\
F_3= \{\{a,b\}\subseteq  A: a+b\in [n]\setminus  A, \ b-a\in [n]\setminus  A\},~
F_4= \{\{a,b\}\subseteq  A: a+b>n, \ b-a\in [n]\setminus A\}.
\end{gather*}
Clearly, $|F|=\sum_{i=1}^{4}|F_i|$ and $|F_1|\leq |[n]\setminus  A|=\xi n$.
Since every $\{a,b\}\in F_2$ satisfies $b=2a\leq n$ and $a+b=3a>n$, we have $|F_2|\leq n/6$. Moreover, we obtain that $|F_3|\leq \binom{\xi n}{2}\leq {{\xi}^2n^2/2}$, as $a+b\in[n]\setminus A$ and $b-a\in[n]\setminus  A$. 
Similarly, we have $|F_4|< \frac{n}{2}\cdot\xi n=\frac{\xi}{2}  n^2$, as $b>n/2$ and $b-a\in[n]\setminus  A$. 
Finally, we conclude that $|F|\leq \xi n+n/6+{\xi}^2n^2/2+\xi n^2/2
\leq \xi n^2+n/6$.
\end{proof}

For a good $r$-template $P$ of $A$, let 
\[
X_1=\{x\in A: |P(x)|=1\},\quad X_2=\{x\in A :  |P(x)|=2\},\quad X_3=\{x\in A : |P(x)|\geq3\},
\] and $x_i=|X_i|$ for $i\in[3]$. 

\begin{lemma} \label{x_3}
Suppose that $P$ is a template of $A$ in $\mathcal{C}_2$. Then we have \[
\max\left\{\frac {(\xi- \delta)n+x_1}{\log{r}-1},\ 0\right\}< 
x_3 \leq 2n^{-1/3}n,
\]
and $\xi< 3(\log r -1)n^{-1/3} + \delta$.
\end{lemma}

\begin{proof}
By the definitions of $G(P,A)$ and $\mathcal{C}_2$, we have
\begin{equation}\label{ineq:courainc}
2^{x_2}r^{x_3} \geq g(P,A)>2^{(1-\delta)n}.
\end{equation}
Since $x_2=|A|-x_1-x_3$ and $|A|=(1-\xi)n$, we obtain that 
\begin{equation}\label{eq:x3lb}
x_3>\frac {(\xi- \delta)n+x_1}{\log{r}-1}.
\end{equation}

We first claim that $x_2\geq (1-\xi)n/3$. Otherwise, we immediately have $x_1+x_3 > 2(1-\xi)n/3$. 
Together with~(\ref{eq:x3lb}), we obtain that
$x_3 > \frac{(2+\xi-3\delta)n}{3\log r} > \frac{n}{2\log r}$. 
By Lemma \ref{(x,y)sumfree} and $0\leq \xi\leq r^{-3}$, there are at least 
\[\binom{x_3}{2}-({\xi}n^2 + n/6)\geq \frac{n^2}{16\log^2 r}\]
pairs in $X_3$ that are contained in some restricted Schur triples in $A$. 
This contradicts the definition of good $r$-templates, as
$RS(P)\geq n^2/(3\cdot 16\log^2 r)> n^{-1/3}s([n])$.

For each $a\in X_3$, let $B_a=\{b\in X_2: \{a,b\}\subseteq S, \text{ for some} \ S\in \mathcal{S}(A)\}$.
Note that every $b\in X_2 \setminus B_a$ satisfies either $|b -a|\in [n]\setminus A$ or $|b - a| = \min\{a, b\}$. 
Then we have
$|B_a|\geq (1-\xi)n/3-2\xi n-2>n/6$.
Since $P$ is a good $r$-template of $A$, we obtain that
$$n^{-1/3}s([n])
\geq RS(P)
\geq \frac12\sum_{a\in X_3}|B_a|\geq \frac{x_3\cdot n}{12},$$
which indicates $x_3\leq 3n^{-1/3}n$.  Moreover, by inequality (\ref{eq:x3lb}), we have $\xi< 3(\log r -1)n^{-1/3} + \delta$.
\end{proof}

Note that if $\xi\ge 3(\log r -1)n^{-1/3} + \delta$, then $\mathcal{C}_2$ is empty by Lemma~\ref{x_3}. 
Therefore, in the following two lemmas, we require that $0\le \xi\le \min\{3(\log r -1)n^{-1/3} + \delta, r^{-3}\}$.
Next, we will prove a stability result on good templates with many rainbow sum-free colorings.

\begin{lemma} \label{(i,j)palette number} 
For every $P\in \mathcal{C}_2$, there exist two colors $\{i,j\}\subseteq [r]$ such that the number of elements in $A$ with palette $\{i,j\}$ is at least $(1-2\delta)n$.
\end{lemma}

\begin{proof}
By Lemma~\ref{x_3} and~(\ref{ineq:courainc}), we have \[x_2\geq(1-\delta-3\log r \cdot n^{-1/3})n.\]
For $1\leq i<j\leq r$, define $Y_{i,j}:=\{x\in X_2 : P(x)=\{i,j\}\}$.
Without loss of generality, we can assume that $|Y_{1,2}|\geq x_2/\binom{r}{2}$. 
Let $Y'=X_2\setminus Y_{1,2}$.
For each $a\in Y'$, let $B_a=\{b\in Y_{1,2}: \{a,b\}\subseteq S, \text{ for some } S\in \mathcal{S}(A)\}$. Similarly as in Lemma~\ref{x_3}, we obtain that $|B_a|\geq x_2/\binom{r}{2}-2\xi n-2$, and then
\[
n^{-1/3}s([n])\geq RS(P)\geq \frac{1}{2}\sum_{a\in Y'}|B_a|\geq \frac{|Y'|\cdot n}{2r(r-1)}.
\]
Since $\delta =1/(24\log n)$, we have $|Y_{1,2}|=x_2-|Y'|\geq (1-2\delta)n$, which completes the proof.
\end{proof}

\begin{lemma} \label{P_{i,j}}
For two colors $\{i,j\}\subseteq [r]$, denote by $\mathcal{P}=\mathcal{P}(i, j)$ the set of good $r$-templates of $A$ in which there are at least $(1-2\delta)n$ elements in $A$ with palette $\{i,j\}$. Then
\[
]g(\mathcal{P}, A)\leq 2^{|A|}(1 + 2^{-n/15}).
\]
\end{lemma}

\begin{proof}

For an $r$-coloring $g\in G(\mathcal{P}, A)$, let $S(g)$ the set of elements in $A$ that are not colored by $i$ or $j$. 
By the definition of $\mathcal{P}$, we have $|S(g)|\leq 2\delta n$.
Define 
\[\mathcal{G}_0=\{g\in G(\mathcal{P}, A) : S(g)=\varnothing\}, \quad\text{and}\quad\mathcal{G}_1=\{g\in G(\mathcal{P}, A) : |S(g)|\geq 1\}.\] Clearly, we have $g(\mathcal{P}, A)=|\mathcal{G}_0| + |\mathcal{G}_1|$ and $|\mathcal{G}_0|\leq 2^{|A|}$. It remains to show that $|\mathcal{G}_1|\leq 2^{|A| - n/15}.$

Let us consider the ways to color $A$ so that the resulting colorings are in $\mathcal{G}_1$. We first choose a set $A_0\subseteq A$ of size at most $2\delta n$, which will be colored by the colors in $[r]\setminus \{i, j\}$. The number of options is at most $\sum_{1\leq k\leq 2\delta n}\binom{n}{k}$, and the number of colorings is at most $r^{2\delta n}$.
Once we fix $A_0$ and the colors of its elements, take an arbitrary element $t\in A_0$. 
\begin{claim}\label{claim:dispair}
Let $\mathcal{D}(t)$ be the collection of disjoint pairs $\{a, b\}$ in $A\setminus A_0$ such that $\{a, b, t\}$ forms a restricted Schur triple.
Then $|\mathcal{D}(t)|\geq n/7$.
\end{claim}
\begin{proof}
Define
\[
\mathcal{S}_1= \{\{a,b\}\subseteq [n]: a + b=t,\ a<b\},\ \text{and} \
\mathcal{S}_2= \{\{a,b\}\subseteq [n]: t+a=b,\ t<a<b\}.
\]
We first observe that $|\mathcal{S}_1|= \lfloor(t-1)/2\rfloor$ for every $t\in [n]$. Note that all pairs in $\mathcal{S}_1$ are disjoint. Therefore, if $t> 2n/5$, we have
$|\mathcal{D}(t)|\geq |\mathcal{S}_1| - \xi n - |A_0|\geq |\mathcal{S}_1| - (2\delta + \xi) n\geq n/7$.
If $t\leq 2n/5$, observe that $|\mathcal{S}_2|= n -2t$ and all pairs in $\mathcal{S}_2$ are disjoint. Therefore, we obtain that $|\mathcal{D}(t)|\geq |\mathcal{S}_2| - \xi n - |A_0|\geq |\mathcal{S}_2| - (2\delta + \xi) n\geq n/7$.
\end{proof}

For every pair $\{a, b\}\in \mathcal{D}(t)$, since $t$ is colored by some color in $[r]\setminus\{i, j\}$, and $a, b$ can only be colored by $i$ or $j$, the elements $a$ and $b$ must receive the same color in order to avoid the rainbow Schur triple. Therefore, together with Claim~\ref{claim:dispair}, the number of ways to finish the colorings is at most
\[
2^{|A| - |A_0| - |\mathcal{D}(t)|}\leq 2^{|A| - n/7}.
\]
Hence, we obtain that
\[
|\mathcal{G}_2|\leq 
\sum_{1\leq k\leq 2\delta n}\binom{n}{k}r^{2\delta n}2^{|A| - n/7}
\leq 2^{|A| + 2\delta n(\log n + \log r) - n/7}\leq 2^{|A| - n/15},
\]
where the last inequality follows from $\delta=1/(24\log n)$.
\end{proof}

Now we have all ingredients to prove Theorem \ref{thm: hrange}.

\begin{proof}[{\bf Proof of Theorem \ref{thm: hrange}}]
First, by property (i) of Theorem~\ref{container} , every rainbow sum-free $r$-coloring of $A$ is a subtemplate of some
$P \in  \mathcal{C}$.
By Property (iii) of Theorem \ref{container} and the definition of $\mathcal{C}_1$ (see~(\ref{def:coll})), we have
$$g(\mathcal{C}_1, A) \leq |\mathcal{C}_1|\cdot 2^{(1-\delta)n} \leq | \mathcal{C}|\cdot 2^{(1-\delta)n} <2^n\cdot 2^{-n/(25\log n)}.$$
If $\xi\geq 3(\log r -1)n^{-1/3} + \delta$, using Lemma~\ref{x_3}, we are done by $g(\mathcal{C}, A)=g(\mathcal{C}_1, A)$.
Otherwise, by Lemma \ref{(i,j)palette number} and Lemma \ref{P_{i,j}}, we obtain that
$$g(\mathcal{C}_{2}, A)= \sum_{1\leq i<j\leq r}g(\mathcal{P}(i,j), A)\leq \binom{r}{2}2^{|A|}(1+2^{-n/15}).$$
Hence, we have
$$g(\mathcal{C}, A)=g(\mathcal{C}_{1}, A)+g(\mathcal{C}_{2}, A)\leq 2^n\cdot 2^{-\frac{n}{25\log n}}+\binom{r}{2}2^{|A|}(1+2^{-n/15})\leq \binom{r}{2}\cdot 2^{|A|}+2^{-\frac{n}{26\log n}}2^n,$$
which gives the desired upper bound on the number of rainbow sum-free $r$-colorings of $A$.
\end{proof}

\section{Proof of Theorems~\ref{thm: three} and \ref{thm:sta}}
The following lemma gives us a structural description of large sum-free sets.
\begin{lemma}\label{lem:str}
Let $\varepsilon,c>0$, $10\varepsilon<c\le 1$, and $\varepsilon<1/10$. Let $A,B\subseteq[n]$ such that $A\cap B=\varnothing$, $B$ is sum-free, and $|A|=cn$. If $|B|\geq  \lceil(1/2 - \varepsilon) n \rceil$, then $(A+B)\cap B\neq \varnothing$.
\end{lemma}

\begin{proof}
Suppose, to the contrary, that $(A+B)\cap B=\varnothing$. Since $|B|\ge  \lceil(1/2 - \varepsilon) n \rceil> \lceil 2n/5 \rceil$, by Lemma~\ref{Stabilitylem3}, either $B$ only contains odd numbers, or $ \min(B)\ge |B|\geq \lceil(1/2 - \varepsilon) n \rceil$. 

If $B$ only contains odd numbers, then there is $d\geq c-\varepsilon$, such that $|A\cap E|=dn$, where $E\subseteq [n]$ is the set of all even numbers. 
Thus, there exists an $a\in A\cap E$ such that $dn\leq a\leq (1-d)n$. 
Let $P$ be the collection of all pairs $\{i,\ i+a\}$, where $i$ is odd, and $1\leq i\leq dn$. Observe that all pairs in $P$ are pairwise disjoint and $P$ contains between $\lfloor dn/2 \rfloor$ and $\lceil dn/2 \rceil$ elements, that is, approximately $dn/2$ elements.
Since $(A+B)\cap B=\varnothing$, for each pair $\{i,j\}$ in $P$, at least one of $\{i,j\}$ is not in $B$. This implies 
\[
|B|\leq \frac n2- |P|\leq \frac n2 - \frac {dn}{2}
\leq \frac n2 - \frac{(c - \varepsilon)n}{2}\leq
\left(\frac{1}{2}-2\varepsilon\right)n,\] which contradicts the assumption of $B$.

Let $b:= \min(B)$. If $b\ge \lceil(1/2 - \varepsilon) n \rceil$,  then there is $d\geq c-2\varepsilon$ such that $|A\cap[b-1]|=dn$. 
Therefore, we have 
$|(A + B) \cap [b,n]| \geq |((A \cap [b-1]) + \{b\}) \cap [b,n]| \geq dn - 2 \varepsilon n > 6 \varepsilon n$.
Since $(A + B) \cap B = \varnothing$, $|B| \leq |[b,n]| - 6\varepsilon n \leq n/2 - 5 \varepsilon n$ which is a contradiction.
\end{proof}

Our next lemma says that when the number of colorings $r=3$ and the size of $A$ is significantly smaller than $n$, the number of rainbow sum-free $r$-colorings of $A$ will be much less than $2^n$. And when $r\geq 5$ and the size of $A$ is significantly larger than $n/2$,  the number of rainbow sum-free $r$-colorings of $A$ will be much less than $r^{n/2}$.

\begin{lemma}\label{lem:mid}
Let $\varepsilon>0$, $r$ be a positive integer, and let $A$ be a subset of $[n]$. Then the following holds.\smallskip
\begin{compactenum}[\rm (i)]
  \item If $r=3$, and $|A|\leq (1-\varepsilon)n$, then there is a constant $\delta_1=\delta_1(\varepsilon)>0$, such that $$g(A,3)\leq 2^{(1-\delta_1)n}.$$
  \item If $r\geq5$, and $|A|\geq \big(1/2+\varepsilon\big)n$, then there is a constant $\delta_1=\delta_1(\varepsilon,r)>0$ such that $$g(A,r)\leq r^{(1/2-\delta_1)n}.$$
\end{compactenum}
\end{lemma}
\begin{proof}
Let $\mathcal{C}$ be the collection of $r$-templates given by Theorem~\ref{container}, and let 
\[
g_{\max}(\mathcal{C},A)=\max_{P\in \mathcal{C}}g(P, A).
\]
For each $P\in\mathcal{C}$, if $P$ is not a good $r$-template of $A$, then $g(P, A)=0$. Therefore, $g_{\max}(\mathcal{C},A)$ is always achieved by a good $r$-template.
Let $P$ be a good template of $A$. By Proposition~\ref{numberofsum}, we have $RS(P)\le n^{-1/3}s([n])=o(n^2)$ which means that the number of rainbow Schur triples induced by $P$ is $o(n^2)$. Let $S\subseteq [n]$ be the set formed by those elements of rainbow Schur triples induced by $P$. Clearly, $S$ is a set containing $o(n^2)$ Schur triples. By Lemma~\ref{Stabilitylem2}, 
there is a set $E\subseteq [n]$ such that $S\setminus E$ is sum-free and $|E|=o(n)$. Therefore, 
there is a template $P':[n]\setminus E\to 2^{[r]}$, such that $P\mid_{[n]\setminus E}=P'$, 
and $RS(P')=0$.
Define 
\[
X_1=\{a\in A\setminus E: |P'(a)|=1\},~ X_2=\{a\in A\setminus E: |P'(a)|=2\},~ X_3=\{a\in A\setminus E: |P'(a)|\geq 3\},
\] 
and $x_i=|X_i|$ for $i\in[3]$. 
Set $T=X_2\cup X_3$. Therefore, we have
\begin{equation}\label{eq:T}
(X_3+X_3)\cap X_3=\varnothing,\quad (T+T)\cap X_3=\varnothing,\quad (X_3+T)\cap T=\varnothing.
\end{equation}
 As $X_3$ is restricted sum-free, we have $x_3\leq \lfloor n/2 \rfloor+1$. 


Let $m$ be the largest element in $X_3$. By (\ref{eq:T}), for every $i< m$, at least one of $\{i,\ m-i\}$ is not in $T$, and same for $X_3$. Hence, we have 
\begin{equation}\label{eq:P3}
|T|\leq n-\Big\lceil\frac{m-1}{2}\Big\rceil,\quad x_3\leq m-\Big\lceil\frac{m-1}{2}\Big\rceil.
\end{equation}

\noindent\textbf{Case 1: $r=3$.}

Observe that we may assume $\varepsilon< 2/5$, as otherwise we will get $g(A, 3)\leq 3^{|A|}\leq 2^{0.955n}$, which completes the proof with $\delta_1=0.045$. We first consider the case when $x_2\leq(1-5\varepsilon/2)n$. 
Then we have
\begin{align*}
    \log g(\mathcal{C},A)&\leq
    \log (|\mathcal{C}|\cdot g_{\max}(\mathcal{C},A))
    \leq cn^{2/3}\log^2n+|E|\log 3+x_2+x_3\log 3\nonumber\\
    &=o(n)+x_2+x_3\log 3
    = o(n)+\left(|T|+x_3-\left(1-\frac{2}{\log 3}\right)x_2\right)\frac{\log 3}{2} \nonumber\\
    &\leq n+o(n)-\frac{5\varepsilon}{2}\left(1-\frac{\log3}{2}\right)n< (1-\delta_1)n,
\end{align*}
where we take $\delta_1=\frac{5\varepsilon}{4}(1-\frac{\log3}{2})$.

Now, we may assume that $x_2> (1-5\varepsilon/2)n$. Then $x_3\leq |A|-x_2<3\varepsilon n/2$. Thus we obtain
\begin{align*}
\log g(\mathcal{C},A)&\leq o(n)+x_2+x_3\log 3
\leq o(n)+|A|+(\log 3-1)x_3\\
&\leq n+o(n)-\frac{\varepsilon}{2}\left(5-3\log3\right)n
<(1-\delta_1)n,    
\end{align*}
and we take $\delta_1=\frac{1}{4}\left(5-3\log3\right)\varepsilon$.\smallskip
\newline

\noindent\textbf{Case 2: $r\geq5$.}

Since $|A|\geq (1/2+\varepsilon)n$, and $P$ is a good $r$-template of $A$, we have that 
$x_1+x_2 = |A|-x_3-|E|\geq \varepsilon n/2$ 
for large enough $n$. We first assume that $x_2\geq\frac{\varepsilon n}{100}$. 
Similarly, we get
\begin{align}
    \log g(\mathcal{C},A)&\leq
    \log (|\mathcal{C}|\cdot g_{\max}(\mathcal{C},A))
    \leq c n^{2/3}\log^2n+|E|\log r+x_2+x_3\log r\nonumber\\
    &=o(n)+x_2+x_3\log r
    = o(n)+\left(|T|+x_3-\left(1-\frac{2}{\log r}\right)x_2\right)\frac{\log r}{2} \nonumber\\
     &\leq o(n)+\left(n-\left(1-\frac{2}{\log r}\right)\frac{\varepsilon}{100}n\right)\frac{\log r}{2} < \left(\frac{n}{2}-\delta_1 n\right)\log r,\label{eq:mid}
\end{align}
where we take $\delta_1=\frac{1}{300}\left(1-\frac{2}{\log r}\right)\varepsilon$. Note that $\delta_1>0$ as $r\geq5$.

Finally, we may assume that $x_2 < \frac{\varepsilon n}{100}$, and then $x_1\geq \frac{\varepsilon n}{2} - x_2\geq \frac{\varepsilon n}{3}$.
We claim that $x_3\leq(1/2-\varepsilon/40)n$. 
Otherwise, by the way we construct $P'$, we also have
$(X_1+X_3)\cap X_3=\varnothing$
and this contradicts Lemma~\ref{lem:str}.
Similarly as before, we can conclude that
\[
\log g(\mathcal{C},A)\leq o(n)+x_2+x_3\log r
\leq o(n)+ \frac{\varepsilon n}{100}+ (1/2-\varepsilon/40)n\log r\leq \left(\frac{n}{2}-\delta_1 n\right)\log r,
\]
where we take $\delta_1=\frac{\varepsilon}{400}$.
\end{proof}

The case when $r=4$ is more involved, and we will discuss it later in this section. But the result in Lemma~\ref{lem:mid} (i) is enough to imply Theorem~\ref{thm: three}.
\begin{proof}[{\bf Proof of Theorem~\ref{thm: three}}]
Observe that $g([n], 3)\geq 3\cdot 2^n -3$. Suppose $A\subseteq[n]$ and $A\neq[n]$. When $|A|\leq (1-3^{-3})n$, by Lemma~\ref{lem:mid} (i), there is $\delta_1>0$ such that $g(A,3)\leq 2^{(1-\delta_1)n}<g([n],3).$ Now, we have $(1-3^{-3})n<|A|\leq n-1$. By Theorem~\ref{thm: hrange}, $g(A,3)\leq (1.5+o(1))2^n<g([n],3)$.
\end{proof}


The next lemma records an easy fact about intervals which is used many times in the later proofs. 

\begin{lemma}\label{lem:inter}
Let $\varepsilon>0$, and let $a,b$ be integers such that $0<a<b<n$, $3\varepsilon n<a<n/2-2\varepsilon n$, and $a+3+3\varepsilon n<b\leq 2a$.
Suppose $A\subseteq [a+1,b]$, $B\subseteq [b+1,n]$, and $|A|>b-a-\varepsilon n$, $|B|>n-b-\varepsilon n$. Then $(A+A)\cap B\neq\varnothing$.
\end{lemma}
\begin{proof}
Let $\alpha$ be the smallest element in $A$, then $\alpha\leq a+\varepsilon n+1$. Let $J=[\alpha+1, \alpha-1+\lceil 2\varepsilon n\rceil]\subseteq [a+1,b]$. Observe that $\alpha+J\subseteq [b+1,n]$. Since $|J|=\lceil2\varepsilon n\rceil-1$,  $|[a+1,b]\setminus A|{\leq  \lceil\varepsilon n\rceil-1}$, $|[b+1,n]\setminus B|{\leq   \lceil\varepsilon n\rceil-1}$,  and $2\lceil\varepsilon n\rceil\leq  \lceil2\varepsilon n\rceil+1$, this implies that there is $\beta\in A\cap J$ such that $\alpha+\beta\in B$.
\end{proof}

The next lemma, Lemma~\ref{lem:mid4}, is similar to Lemma~\ref{lem:mid}, but here we consider the case when the number of colors is $4$. In order to obtain the same conclusion as in Lemma~\ref{lem:mid}, we further require that the size of $A$ is significantly smaller than $n$, since if $A$ is close to $[n]$, when we color all the elements in $A$ by two colors, the number of colorings we obtained is also close to the extremal case. Note that if we use the same proof as in Lemma~\ref{lem:mid} for $r=4$, equation (\ref{eq:mid}) does not give us the conclusion we want. Hence the proof of Lemma~\ref{lem:mid4} requires a more careful and complicated analysis of the structures of the containers.

\begin{lemma}\label{lem:mid4}
Let $0<\varepsilon<1/21$ and $A\subseteq [n]$ with $ \big(1/2+\varepsilon\big)n\leq |A|\leq (1-\varepsilon)n$. Then there is $\delta_2=\delta_2(\varepsilon)>0$ such that 
\[
g(A,4)\leq 4^{n/2-\delta_2n}.
\]
\end{lemma}
\begin{proof}
Let $\mathcal{C}$ be the collection of containers obtained from Theorem~\ref{container}, and $P\in\mathcal{C}$ be a good $r$-template of $A$. 
Similarly to what we did in the proof of Lemma~\ref{lem:mid}, applying Lemma~\ref{Stabilitylem2} to $P$, we obtain a template $P':[n]\setminus E\to 2^{[r]}$, such that $P\mid_{[n]\setminus E}=P'$, $|E|=o(n)$, and  $RS(P')=0$.
We now partition $A\setminus E$ into three sets:
let $X_1=\{a\in A\setminus E : |P'(a)|=1\}$, $X_2=\{a\in A\setminus E :  |P'(a)|=2\}$, and $X_3=\{a\in A\setminus E :  |P'(a)|\geq 3\}$.  We use $T$ to denote the set $X_2\cup X_3$, and write  $x_i=|X_i|$ for $i=1,2,3$.
Similarly to (\ref{eq:T}), we have
\begin{equation}\label{eq:T:lem44}
(X_3+X_3)\cap X_3=\varnothing,\quad (T+T)\cap X_3=\varnothing,\quad (X_3+T)\cap T=\varnothing;
\end{equation}
moreover, $X_3$ is restricted sum-free with size 
\begin{equation}\label{eq:x3ub:lem44}
x_3\leq \lfloor n/2 \rfloor+1.
\end{equation}
Since $|A|\geq(1/2+\varepsilon)n$, and $P$ is a good template, the upper bound on $x_3$ implies that 
\begin{equation}\label{eq:x12bound:lem44}
x_1+x_2\geq \varepsilon n/2. 
\end{equation}
Let $m$ be the largest element in $X_3$, and similarly to~(\ref{eq:P3}) we have
$|T|\leq n-\lceil(m-1)/2\rceil$ and $x_3\leq m-\lceil(m-1)/2\rceil.$

We claim that it suffices to show 
\begin{equation}\label{eq:mainclaim:lem44}
|T|\leq n-\Big\lceil\frac{m-1}{2}\Big\rceil-\frac{\varepsilon n}{2000},\quad\text{ or }\quad x_3\leq m-\Big\lceil\frac{m-1}{2}\Big\rceil-\frac{\varepsilon n}{2000}.
\end{equation}
Once we have~\eqref{eq:mainclaim:lem44}, by Theorem~\ref{container} we immediately have
\begin{align*}
 g(\mathcal{C},A)\leq 2^{cn^{2/3}\log^2n}r^{|E|}2^{x_2}r^{x_3}=r^{o(n)+\frac{1}{2}(|T|+x_3)}\leq r^{\frac{n}{2}+o(n)-\frac{\varepsilon n}{2000}}<r^{\frac{n}{2}-\delta_2n},
\end{align*}
where we take $\delta_2=\varepsilon/3000.$ 

Assume by contradiction that
\begin{equation}\label{eq:P3r=4}
    n-\Big\lceil\frac{m-1}{2}\Big\rceil-\frac{\varepsilon n}{2000}< |T|\leq n-\Big\lceil\frac{m-1}{2}\Big\rceil,\text{ and }\ m-\Big\lceil\frac{m-1}{2}\Big\rceil-\frac{\varepsilon n}{2000}< x_3\leq m-\Big\lceil\frac{m-1}{2}\Big\rceil.
\end{equation}
For simplicity, let $\alpha=n-m$ (recall that $m$ is the largest element in $X_3$).
If $\alpha\leq\varepsilon n/40$, then \eqref{eq:P3r=4} leads to $|X_3|\geq (1/2 - \varepsilon)n$.
Thus by \eqref{eq:x12bound:lem44} and Lemma~\ref{lem:str}, at least one of 
\[
(X_1+X_3)\cap X_3\neq\varnothing, \quad (X_2+X_3)\cap X_3\neq\varnothing,
\]
is true, which contradicts the fact that $RS(P')=0$.
Therefore we may assume that 
\begin{equation}\label{eq:alphalb:lem44}
\alpha >\varepsilon n/40. 
\end{equation}
On the other hand, as $|A|\leq (1-\varepsilon)n$, the lower bound on $|T|$ in (\ref{eq:P3r=4}) gives us
\[
n-\Big\lceil\frac{m-1}{2}\Big\rceil-\frac{\varepsilon n}{2000}<(1-\varepsilon)n
\]
and hence $m> 3\varepsilon n/2$.

We  now partition $[n]$ into three intervals $J_1,J_2,J_3$, such that 
$$J_1=[1,\alpha],\quad J_2=[\alpha+1,n-\alpha],\quad \text{ and } J_3=[n-
\alpha+1,n].$$
By the definition of $\alpha$, we see that $X_3\subseteq J_1\cup J_2$. 
Note that $(T+X_3)\cap X_3=\varnothing$ and $n-\alpha= m\in X_3$. 
Then for every $1\leq k<m$, at most one of the elements in $\{k,m-k\}$ is in $T\cap(J_1\cup J_2)$, and thus $|T\cap(J_1\cup J_2)|\leq \lceil m/2\rceil$. Then by~\eqref{eq:P3r=4}, we have
\begin{equation}\label{eq:P2r=4}
|J_3\cap X_2|=|J_3\cap T|=|T| - |T\cap(J_1\cup J_2)|\geq  \alpha-\frac{\varepsilon n}{1000}. 
\end{equation}

The rest of the proof will be processed by the following six steps.\\
\newline
\noindent\textbf{Step 1: show $|J_1\cap X_3|< dn$, where $d=\varepsilon/400$. }
\begin{proof}
Assume by contradiction that $|J_1\cap X_3|\geq dn$, then we can find $\beta\in X_3$ such that $dn/2\leq \beta\leq \alpha-dn/2$. We now consider the subinterval 
$$
J_3'=\left[n-\alpha+1,n-\alpha+\frac{dn}{2}\right]\subseteq J_3.
$$ 
Note that by the choice of $\beta$, we must have $(J_3'+\beta)\cap J_3'=\varnothing$, and $J_3'+\beta\subseteq J_3$.
Since $(X_2+X_3)\cap X_2=\varnothing$, for every element $t\in J_3'$, at least one element in the pair $\{t, t+\beta\}$ is not contained in $X_2$. Thus we have that 
$$
|J_3\cap X_2|\leq |J_3| - |J'_3|=\alpha-\frac{dn}{2}=\alpha-\frac{\varepsilon n}{800},
$$
which contradicts (\ref{eq:P2r=4}). 
\end{proof}
Now using (\ref{eq:P3r=4}) and Step 1, we get
\begin{equation}\label{eq: bound on J_2 cap X_3}
|J_2\cap X_3|=x_3-|J_1\cap X_3| \geq \frac{n-\alpha}{2}-\frac{\varepsilon n}{1000}-dn\geq \frac{n-\alpha}{2} - \frac{\varepsilon n}{250}.
\end{equation}

\noindent\textbf{Step 2: show $\alpha\leq n/3.$}
\begin{proof}
Again, assume by contradiction that $\alpha>n/3$. Then \eqref{eq: bound on J_2 cap X_3} implies 
$$
|J_2\cap X_3|
\geq (n-2\alpha) -\frac{\varepsilon n}{250}.$$ 
Applying (\ref{eq:P2r=4}) and Lemma~\ref{lem:inter} (with $A=X_3
\cap J_2$, $B=X_2
\cap J_3$, $a=\alpha$, and $b=n-\alpha$), we get $(X_3+X_3)\cap X_2\neq\varnothing$, which contradicts $RS(P')=0$.
\end{proof}
We then study the structure of $X_3$ inside $J_2$. 
Let $J_2':=[n-2\alpha +1,n-\alpha]$. Note that by Step 2, $J_2'\subseteq J_2$. \\
\newline
\noindent\textbf{Step 3: show $|J_2'\setminus X_3|\leq dn+\varepsilon n/800=3\varepsilon n/800.$}
\begin{proof}
Observe that at least one of $\{t,m-t\}$ is not in $X_3$; then by Step 2 and the range of $J_2\setminus J_2'$, we have $|X_3\cap (J_2\setminus J_2')|\leq(n-3\alpha)/2$. If $|J_2'\setminus X_3|> dn+\varepsilon n/800$, then together with Step 1,
\[
x_3\leq \frac{n-3\alpha}{2}+dn+\alpha-\Big(dn+\frac{\varepsilon n}{800}\Big)\leq \frac{m}{2}-\frac{\varepsilon n}{800},
\]
which contradicts (\ref{eq:P3r=4}). 
\end{proof}

\noindent\textbf{Step 4: show $|(J_1+\alpha)\cap X_3|\leq d'n$, where $d'=\varepsilon/60$.}
\begin{proof}
Assume by contradiction that $|(J_1+\alpha)\cap X_3|\geq d'n$. Then there is $\gamma\in X_3$ such that $\alpha+d'n/2\leq \gamma\leq 2\alpha-d'n/2$. Define
\[
J_2'':=\left[n-2\alpha +1,n-2\alpha+\frac{d'n}{2}\right] \stackrel{\eqref{eq:alphalb:lem44}}{\subseteq} J_2'.
\]
Note that as $\gamma\in X_3$, we have $(X_2\cap J''_2+\gamma)\cap X_2=\varnothing$; moreover, observe that $\gamma+J_2''\subseteq J_3$. Then we have $|X_2\cap J_3|\leq |J_3| - |X_2\cap J''_2|$, and thus
\[
|X_2\cap J_3|+ |X_2\cap J'_2| \leq |X_2\cap J_3|+ |X_2\cap J''_2| + \alpha - \frac{d'n}{2} \leq 2\alpha - \frac{d'n}{2}.
\]
Then at least one of $|X_3\cap J_2'|$ and $|X_2\cap J_3|$ is less than $\alpha-d'n/4=\alpha-\varepsilon n/240$, which contradicts either (\ref{eq:P2r=4}) or Step 3. 
\end{proof}

\noindent\textbf{Step 5: show that
\begin{equation}\label{eq: bound on X3 cap mid interval}
\frac{n-3\alpha}{2}-\frac{\varepsilon n}{1000}\leq \big|X_{3}\cap [\alpha+1,n-2\alpha]\big|\leq n-4\alpha+d'n.
\end{equation}
}
\begin{proof}
Note that $J_2'$ can be rewritten as $J_2'=\{m\}\cup (m-J_1\setminus\{\alpha\})$. Since $X_3$ iteself is restricted sum-free, for any $a\in J_1\setminus\{\alpha\}$, at least one of $\{a, m-a\}$ is not in $X_3$, and therefore $|X_3\cap (J_1\cup J_2')|\leq |J_1\cup J_2'| -( |J_1|-1)\leq\alpha+1$. 
Then the lower bound immediately follows from~\eqref{eq:P3r=4}. 

Recall from Step 4 that $|X_3\cap (J_1+\alpha)|\leq d'n$. To see the upper bound, it suffices to prove that  $J_1+\alpha=[\alpha+1, 2\alpha]\subseteq [\alpha+1, n-2\alpha]$.
Indeed, if not, we would have $|X_3\cap [\alpha+1,n-2\alpha]|\leq d'n $, and together with the lower bound of~\eqref{eq: bound on X3 cap mid interval}, we conclude that 
\[
\alpha\geq \frac{n}{3}-\frac{\varepsilon n}{1500}-\frac{2d'n}{3}. 
\]
Then combining with Step 1 and Step 4, we have
\begin{align*}
x_3&= |X_3\cap J_1|+ |X_3\cap (J_1+\alpha)|+|X_3\cap [2\alpha+1, n-\alpha]|\\
&\leq dn+d'n+n-3\alpha <\frac{11\varepsilon n}{200},
\end{align*}
 which is a contradiction to~\eqref{eq:P3r=4}.
Note that when $\varepsilon<1/4$, by \eqref{eq:P3r=4} and Step 2, i.e., $\alpha\leq n/3$, we always have
\[
x_3\geq \frac{n-\alpha}{2}-\frac{\varepsilon n}{2000}\geq \frac{n}{3}-\frac{\varepsilon n}{2000}\geq  \frac{\varepsilon n}{2} > \frac{11\varepsilon n}{200},
\]
and this is a contradiction. 
\end{proof}

Comparing the upper and lower bounds in \eqref{eq: bound on X3 cap mid interval}, we conclude that \begin{equation}\label{eq: n is at least 5 alpha}
    \alpha\leq\frac{n}{5}+\frac{2d'n}{5}+\frac{\varepsilon n}{2500}.
\end{equation}
\noindent\textbf{Step 6: show $\alpha< n/5-\varepsilon n/1000$.}
\begin{proof}
Suppose that $\alpha\geq n/5-\varepsilon n/1000$. First by \eqref{eq:P3r=4}, Step 1 (i.e., $|X_3\cap J_1|\leq dn$), and Step 4 (i.e., $|X_3\cap (J_1+\alpha)|\leq d'n$), we have that 
$$
|X_3\cap (J_2\setminus (J_1+\alpha))|\geq \frac{n-\alpha}{2}-dn-d'n-\frac{\varepsilon n}{2000},
$$
which then implies
\begin{align*}
|[2\alpha +1,n-\alpha]\setminus X_3|&\leq n-3\alpha - \frac{n-\alpha}{2}+dn+d'n+\frac{\varepsilon n}{2000}\\
&\leq dn+d'n+\frac{3\varepsilon n}{1000}. 
\end{align*}
In particular, we have
\begin{equation}\label{cap-X_3}
\Big|\Big[2\alpha +\frac{3\varepsilon n}{1000},n-\alpha\Big]\cap X_3\Big|\geq n-3\alpha-\frac{3\varepsilon n}{1000}-\Big(d+d'+\frac{3\varepsilon}{1000}\Big)n. 
\end{equation}
As we assume $\alpha\geq n/5-\varepsilon n/1000$, we get
\begin{equation}\label{b<2a}
n-\alpha \leq 4\alpha +\frac{5\varepsilon n}{1000}\leq  
2\Big(2\alpha +\frac{3\varepsilon n}{1000}-1\Big). 
\end{equation}
Next we intend to apply
Lemma~\ref{lem:inter} with  $a=2\alpha+\frac{3\varepsilon n}{1000}-1$, $b=n-\alpha$, $A=X_3
\cap [2\alpha + \frac{3\varepsilon n}{1000}, n-\alpha]$ and $B=X_2\cap J_3$. Now we check the conditions stated in Lemma~\ref{lem:inter}.
The condition $3\varepsilon n<a< n/2-2\varepsilon n$ follows from \eqref{eq: n is at least 5 alpha} and the lower bound on $\alpha$, and $b\le 2a$ follows from \eqref{b<2a}. By  \eqref{eq:P2r=4} and \eqref{cap-X_3}, the sets $A$ and $B$ also satisfy the conditions stated in Lemma~\ref{lem:inter}. 
Thus, we obtain that $(X_3+X_3)\cap X_2\neq\varnothing$, and this contradicts (\ref{eq:T:lem44}).
\end{proof}

Finally, from~\eqref{eq:alphalb:lem44} and Step 6, we have
\[
\frac{\varepsilon n}{40}< \alpha<\frac{n}{5}-\frac{\varepsilon n}{1000}.
\]
By (\ref{eq:P3r=4}), we have that $x_3>2n/5$. By Lemma~\ref{Stabilitylem3}, either $X_3$ consists of odd integers, or the minimum element in $X_3$ is at least $\frac{n-\alpha}{2}-\frac{\varepsilon n}{1000}$. The first case is impossible, as otherwise we have
\[
\frac{n-\alpha}{2}-\frac{\varepsilon n}{2000}<x_3=|X_3\cap J_1|+ |X_3\cap J_2|\leq dn+\frac{n-2\alpha}{2},
\]
that is, $\alpha<\frac{\varepsilon n}{200}+\frac{\varepsilon n}{1000}$, which contradicts \eqref{eq:alphalb:lem44}. Now, we assume $a\in X_3$ is the minimum element, i.e., $X_3\subseteq[a, n-\alpha]$, and $a\geq \frac{n-\alpha}{2}-\frac{\varepsilon n}{1000}$. Together with (\ref{eq:P3r=4}), we have
 $$\left|\left[\frac{n-\alpha}{2}+1,n-\alpha\right]\setminus X_3\right|\leq\frac{\varepsilon n}{500}.$$ Using the fact that $\frac{n-\alpha}{2}\leq\frac{n}{2}-\frac{\varepsilon n}{80}$, by (\ref{eq:P2r=4}) and Lemma~\ref{lem:inter} (applied to sets $X_3
\cap [\frac{n-\alpha}{2}+1,n-\alpha]$ and $X_2
\cap J_3$), we have $(X_3+X_3)\cap X_2\neq\varnothing$, which contradicts (\ref{eq:T:lem44}).
\end{proof}

The final lemma considers the case when $A$ contains many Schur triples. 
\begin{lemma}\label{lem:manyschur}
Let $r\geq4$ be an integer. Suppose there is $\mu>0$, such that $s(A)\geq\mu n^2$. Then 
\[g(A,r)\leq r^{|A|-\frac{3(2\log r-\log (3r-2))}{2\log r}\mu n}.
\]
\end{lemma}
\begin{proof}
Since $s(A)\geq\mu n^2$, by pigeonhole principle, there is $t\in A$, such that 
\[
s(t,A)\geq \frac{3\mu n^2}{|A|}\geq 3\mu n.
\]

Now we consider the possible number of rainbow sum-free $r$-colorings of $A$.
Recalling the link graph  $L_t(A)$,  we have $k(t,A)\ge 3\mu n/2$ by (\ref{ineq:mat}). We first fix a maximum matching $M$ of $L_t(A)$. For the elements in $A\setminus V(M)$, we color them arbitrarily. 
For each edge $ab\in E(M)$, in order to avoid a rainbow Schur triple, we either let $a, b$ share the same color, or color one of $a, b$ by the color of $t$, and color the other vertex by a different color.
In this way, $a, b$ have exactly $r+2(r-1)$ effective colorings. Let $k=k(t, A)$. Hence, we have 
\begin{align*}
    g(A,r)&\leq r^{|A|-2k-1}r(3r-2)^k
    \leq r^{|A|}\Big(\frac{3r-2}{r^2}\Big)^{\frac{3\mu n}{2}}=r^{|A|-\frac{3(2\log r-\log (3r-2))}{2\log r}\mu n},
\end{align*}
as desired.
\end{proof}
Now we can prove the stability theorem.
\begin{proof}[{\bf Proof of Theorem \ref{thm:sta}}]
The first part of the statement, that $g(n,r)\leq r^{n/2+o(n)}$, follows easily from the fact $g(A,r)\leq r^{|A|}$ when $|A|\leq n/2+o(n)$. If $|A|\geq n/2+\eta n$ for some constant $\eta$, the result follows from Lemma~\ref{lem:mid} (ii) when $r\geq5$. For the case $r=4$, after applying Lemma~\ref{lem:mid4} we still have one extra case that $|A|\geq (1-\eta)n$, and this follows from Theorem~\ref{thm: hrange}.

For the second part of the statement, we will prove it by contrapositive. Let $c=\frac{3(2\log r-\log(3r-2))}{2\log r}$, clearly $c>0$ when $r\geq4$. Let $\mu$ be the value of $\delta(\frac{\varepsilon}{20})$ given in Lemma~\ref{Stabilitylem2}, and let  $\varepsilon'=\min\{\frac{c\mu}{2},\varepsilon\}$. We first consider $r\geq5$, and suppose that we have both
\begin{equation}\label{eq:r=5}
    |A\triangle O|>\varepsilon n, \quad\text{ and }\quad |A\triangle I_0|>\varepsilon n.
\end{equation}
In this case we take $\delta=\min\{\delta_1(\varepsilon'),\varepsilon',\frac{\varepsilon}{20}\}$, where $\delta_1(\varepsilon')$ is given in Lemma~\ref{lem:mid} (ii).
If $|A|\geq \frac{n}{2}+\varepsilon' n$, we apply Lemma~\ref{lem:mid} (ii) with parameter $\varepsilon'$, then we obtain that $g(A,r)\leq r^{n/2-\delta n}$. Thus we may assume that $|A|\leq (1/2+\varepsilon')n$. If $s(A)\geq\mu n^2$, applying Lemma~\ref{lem:manyschur}, we have
\[
g(A,r)\leq r^{n/2+\varepsilon'n-c\mu n}\leq r^{n/2-\varepsilon' n}\leq r^{n/2-\delta n}.
\]
Finally, we have $s(A)<\mu n^2$. By Lemma~\ref{Stabilitylem2}, we get the partition $A=B\cup C$, where $B$ is sum-free and $|C|<\frac{\varepsilon}{20} n$. Note that we may assume $|A|\geq\frac{n}{2}-\frac{\varepsilon}{20} n$, otherwise $g(A,r)\leq r^{|A|}\leq r^{n/2-\delta n}$. Now we have $$|B|\geq |A|-|C|\geq \frac{n}{2}-\frac{\varepsilon n}{10}\geq\frac{2n}{5}$$ since $\varepsilon\leq1$. We apply Lemma~\ref{Stabilitylem3} to $B$. Hence either $B$ contains only odd integers, or the minimum element of $B$ is at least $|B|$. Suppose $B$ consists of odd integers. Thus
\[
|A\triangle O|\leq |C|+|O\setminus B|\leq \frac{\varepsilon}{20} n+\frac{\varepsilon}{10} n<\varepsilon n,
\]
contradicts (\ref{eq:r=5}). Thus, let $a$ be the minimum element in $B$, then $a\geq \frac{n}{2}-\frac{\varepsilon n}{10}$. Therefore,
\[
|A\triangle I_0|\leq |C|+|B\triangle I_0|\leq\frac{\varepsilon}{20}n+\frac{\varepsilon}{10}n+\frac{\varepsilon}{5}n<\varepsilon n,
\]
which also contradicts (\ref{eq:r=5}).

Next, let us consider the case when $r=4$. Besides (\ref{eq:r=5}), we further require\begin{equation}\label{eq:r=4}
    |A\triangle [n]|>\varepsilon n.
\end{equation}
We now take $\delta=\min\{\delta_2(\varepsilon'),\varepsilon', \frac{\varepsilon}{20}\}$, where $\delta_2(\varepsilon')$ is given in Lemma~\ref{lem:mid4}. The case when $|A|\leq \frac{n}{2}+\varepsilon'n$ is same as when $r\geq5$. When $\frac{n}{2}+\varepsilon'n\leq |A|\leq n-\varepsilon'n$, by applying Lemma~\ref{lem:mid4} we get $g(A,4)\leq4^{n/2-\delta n}$. When $|A|\geq n-\varepsilon' n$, we get $|A\triangle [n]|\leq \varepsilon' n\leq\varepsilon n$, which contradicts (\ref{eq:r=4}).
\end{proof}

\section{Proof of Theorem~\ref{thm: r8}}

\begin{prop}\label{prop:larmat}
Let $n, r, c \in \mathbb{N}$ with $r\geq 8$ and $c>1$. Suppose that $A$ is a subset of $[n]$ of size $\lceil n/2\rceil+c$.\smallskip
\begin{compactenum}[\rm (i)]
  \item If there exists an element $t\in A$ such that $k(t, A)\geq 2(c-1)$, then we have $g(A,r)< r^{\lceil n/2\rceil+1}$.\smallskip
  \item If there exists an element $t\in A$ such that $k(t, A)\geq 2(c-1)+1$, then we have $g(A,r)< r^{\lceil n/2\rceil}\left(3-2/r\right)$.
\end{compactenum}
\end{prop}

\begin{proof}
First note that for $r\geq 8$ we have 
\begin{equation}\label{eq:estr}
\frac{(3r-2)^2}{r^3}<1.
\end{equation}
Let $k=k(t, A)$. Similarly as in the proof of Lemma~\ref{lem:manyschur}, we obtain that
$$g(A,r)\leq r^{\lceil n/2\rceil+c-2k}(3r-2)^k
=r^{\lceil n/2\rceil+c}\left(\frac{3r-2}{r^2}\right)^{k}.$$
For $k\geq 2(c-1)$, we have
$$
g(A,r)\leq r^{\lceil n/2\rceil+c}\left(\frac{3r-2}{r^2}\right)^{2(c-1)}
=r^{\lceil n/2\rceil+1}\left(\frac{(3r-2)^2}{r^3}\right)^{c-1} 
< r^{\lceil n/2\rceil+1},
$$
where the last inequality follows from (\ref{eq:estr}) and $c>1$.

Similarly, for $k\geq 2(c-1)+1$, we have
\[
\begin{split}
g(A,r)&\leq r^{\lceil n/2\rceil} \left(3-\frac{2}{r}\right) r^{c-1}\left(\frac{3r-2}{r^2} \right)^{k-1}
\leq r^{\lceil n/2\rceil} \left(3-\frac{2}{r}\right)r^{c-1}\left(\frac{3r-2}{r^2} \right)^{2(c-1)}\\
&= r^{\lceil n/2\rceil} \left(3-\frac{2}{r}\right)\left(\frac{(3r-2)^2}{r^3}\right)^{c-1}
<  r^{\lceil n/2\rceil} \left(3-\frac{2}{r}\right).
\end{split}
\]
Together with the previous inequality, this completes the proof.
\end{proof}

\begin{lemma}\label{lem:staodd}
Let $r\geq 8$, $0<\varepsilon \leq 1/36$, and $A$ be a subset of $[n]$ of size $\lceil n/2\rceil + c$, where $1 < c \leq \varepsilon n$. Suppose that there exists a partition $A=B\cup C$ such that $B\subseteq O \cap A$ and $|C|\leq \varepsilon n$. Then we have $g(A,r)< r^{\lceil n/2\rceil}\left(3-2/r\right)$. 
\end{lemma}

\begin{proof}
From the assumption of $A$, there must be an even number $t\in A$.
By Proposition~\ref{prop:larmat}(ii) and $\varepsilon \leq 1/36$, it is sufficient to show that $k(t, A)\geq (1/12-\varepsilon)n -1\ge 2c-1$. 

Recall that $O$ is the set of all odd numbers in $[n]$. Since $|A|\geq \lceil n/2\rceil + 1$ and $|C|\leq \varepsilon n$, we have $|O\setminus B|\leq \varepsilon n$.
Then,
\[
k(t, A)\geq k(t, B\cup\{t\})\geq k(t, O\cup\{t\})-\varepsilon n,
\]
and thus it is equivalent to show that $k(t, O\cup\{t\})\geq n/12 - 1$.

If $t\geq n/3$, we immediately have 
\[
k(t, O\cup\{t\})\geq \left|\left\{(i, t-i, t), i\in O\cap\left[t/2-1\right]\right\}\right|\geq t/4-1\geq n/12-1.
\]
If $t < n/3$, then by~(\ref{ineq:mat}) we obtain that
\[
k(t, O\cup\{t\})\geq s(t, O\cup\{t\})/2\geq \left|\left\{(t, i, t+i), i\in O\cap\left[t+1,\  n-t\right]\right\}\right|/2
\geq (n-2t)/4\geq n/12. 
\]
This completes the proof.
\end{proof}
\begin{lemma}\label{lem:staint}
Let $r\geq 8$, $0<\varepsilon <1/n^2$, and $A$ be a subset of $[n]$ of size $\lceil n/2\rceil + c$, where $1 < c \leq \varepsilon n$. Suppose that there exists a partition $A=B\cup C$ such that $B\subseteq I_0 \cap A$ and $|C|\leq \varepsilon n$. Then the following holds.\smallskip
\begin{compactenum}[\rm (i)]
  \item If $n$ is even, then $g(A,r)< r^{\lceil n/2\rceil+1}$.\smallskip
  \item If $n$ is odd, then $g(A,r)< r^{\lceil n/2\rceil}\left(3-2/r\right)$.
\end{compactenum}
\end{lemma}
\begin{proof}
Let $m$ be the minimum element of $A$, and clearly $m \leq\lfloor n/2\rfloor -(c-1)$.
Recall that $I_0=[\lfloor n/2\rfloor +1, n]$.
Let $d=|I_0\setminus B|$. From the assumption of $A$, we have $d\leq \varepsilon n$. 
For simplicity, let  $I_{0}' = I_0 \cup \{m\}$.
We divide the proof into four cases.\smallskip

\noindent\textbf{Case 1: $m\leq d+3(c-1)$.} In this case, we have $m\leq 4\varepsilon n$.
Similar to the proof of Lemma~\ref{lem:staodd}, we have
\[
k(m, A)\geq k(m, I'_0) - d
\geq s(m, I'_0)/2 - \varepsilon n
\geq (n/2 - m)/2 - \varepsilon n
\geq n/4 - 3\varepsilon n,
\]
which, together with Proposition~\ref{prop:larmat}(ii) and $\varepsilon< 1/n$, completes the proof.\smallskip

\noindent\textbf{Case 2: $d+3(c-1)< m \leq \lceil n/2 \rceil-d-3(c-1)$.} 
Since $m\leq n/2$, each nontrivial component of $L_m(I'_0)$ is a path, and there are $\min\left\{m, \lceil n/2\rceil - m\right\}\geq d+3(c-1)$ of them.
Therefore we have
\[
 k(m, A)\geq k(m, I'_0) - d
 \geq d+3(c-1) - d=3(c-1)
\]
which, together with Proposition~\ref{prop:larmat}(ii) and $c>1$, completes the proof.\smallskip

\noindent\textbf{Case 3: $\lceil n/2\rceil-d-3(c-1)< m \leq \lceil n/2\rceil-2(c-1)$.}
By the choice of $m$, each nontrivial component of $L_m(A)$ is a path of length 1, and the number of them is at least
\[
s(m, [m, n])- |[m+1, n]\setminus A|
= n-2m - (n -|A| -(m-1))
= \lceil n/2\rceil +(c-1)-m.
\]
Therefore, we obtain that 
\[
k(m, A)=\lceil n/2\rceil +(c-1)-m\geq 3(c-1), 
\]
which completes the proof together with Proposition~\ref{prop:larmat}(ii) and $c>1$.\smallskip

\noindent\textbf{Case 4: $\lceil n/2\rceil-2(c-1)<m \leq \lfloor n/2\rfloor-(c-1)$.}
Similarly as in Case 3, we obtain that $k(m, A)=\lceil n/2\rceil +(c-1)-m$. By the choice of $m$, for even $n$, we have $k(m, A)\geq 2(c-1)$, while for odd $n$, $k(m, A)\geq 2(c-1)+1$.
By Proposition~\ref{prop:larmat}, this gives the desired upper bounds.
\end{proof}

\begin{proof}[{\bf Proof of Theorem~\ref{thm: r8}.}]
Here we only prove (i) as the proof of (ii) is similar.
If $|A|=\lceil n/2 \rceil +1$ and $A\neq I_2$, then $A$ must have at least one restricted Schur triple, and therefore $g(A, r)<g(I_2, r)=r^{\lceil n/2 \rceil +1}$.
When $|A|>\lceil n/2 \rceil +1$, choose a constant $\varepsilon < 1$, which satisfies the assumptions of Lemmas~\ref{lem:staodd} and~\ref{lem:staint}. Then by Theorem~\ref{thm:sta}, we can further assume that $|A \bigtriangleup O|\leq \eps n$, or $|A\bigtriangleup I_0|\leq \eps n$. Applying Lemmas~\ref{lem:staodd} and~\ref{lem:staint} (i) on $A$, for both cases, we obtain $g(A, r)<r^{\lceil n/2 \rceil +1}$.
\end{proof}


\section{Concluding Remarks}
Our investigation raises many open problems. 
In this paper, we determine the rainbow $r$-extremal sets, that is, the subsets of $[n]$ which maximize the number of rainbow sum-free $r$-colorings, for $r\leq3$ and $r\geq8$. However, for $r\in\{4,5,6,7\}$, although Theorem~\ref{thm:sta} says the rainbow $r$-extremal sets should be close to what we expect, our proofs cannot give the exact structure of the extremal sets. 
Therefore, the most interesting question is to determine the unsolved cases of Conjecture~\ref{conj}. 
Recall that $I_1=[\frac{n}{2}-1,n]$ and $I_3=[\frac{n-1}{2},n]$. 

\begin{conj}\label{conj:concluding}
Let $n, r$ be positive integers  and $4\leq r\leq 7$.\smallskip
\begin{compactenum}[\rm (i)]
\item If $n$ is even, then $g(n, r)=r^{n/2}\left(3 - 2/r\right)^2$, and $I_1$ is the unique rainbow $r$-extremal set.\smallskip
\item If $n$ is odd and $r=4$, then $g(n, r)=g([n], r)$, and $[n]$ is the unique rainbow $r$-extremal set.\smallskip
\item If $n$ is odd and $5\leq r\leq 7$, then $g(n, r)=r^{\lceil n/2\rceil}\left(3 - 2/r\right)$, and $I_3$ is the unique rainbow $r$-extremal set.
\end{compactenum}
\end{conj}

Another direction is that one can consider various generalizations of this problem.
Recall that a sum-free set is a set forbidding the solutions of the linear equation $x_1+x_2=y$. 
It is natural to extend the Erd\H{o}s-Rothschild problems to sets forbidding solutions of other linear equations, for example, the $(k, \ell)$-free sets, that is, the sets without nontrivial tuples $\{x_1, \ldots, x_{k}, y_1,\ldots, y_{\ell}\}$ satisfying $\sum_{i=1}^kx_i = \sum_{j=1}^\ell y_j$. 
It is possible that the method used to prove Theorem~\ref{thm: hrange} can prove the analogous results for some other $(k, \ell)$-free sets. However, the stability analysis on other parts would be very involved.

One could also broaden the study of rainbow Erd\H{o}s-Rothschild problems to various other extremal problems in this fashion. 
In the rainbow Erd\H{o}s-Rothschild problems studied to date, that is, the Gallai colorings and the rainbow sum-free colorings, for $r=3$ the configurations maximizing the number of such colorings are complete graphs or the whole intervals, while for sufficiently large $r$ the optimal configurations are those solving the original extremal problems.  
It would be very interesting to determine the threshold of $r$ to ensure that the extremal configurations for the uncolored problems are optimal for rainbow Erd\H{o}s-Rothschild problems.

\section*{Acknowledgements}
This work was supported by the Natural Science Foundation of China  (12231018) and  Young Taishan Scholars program of Shandong Province (201909001).
Part of the work was done while the second and the third authors were visiting the School of Mathematics at Shandong University. They would like to thank the school for the hospitality they received. In addition,
the authors thank Tuan Tran for pointing out the reference for
Theorem~\ref{thm:staden}, and the anonymous referee for carefully reading the manuscript and providing many helpful comments.

\bibliographystyle{abbrv}
\bibliography{ref}

\begin{thebibliography}{10}

\bibitem{ABKS}
N.~Alon, J.~Balogh, P.~Keevash, and B.~Sudakov.
\newblock The number of edge colorings with no monochromatic cliques.
\newblock {\em J. London Math. Soc. (2)}, 70(2):273--288, 2004.

\bibitem{balogh2006remark}
J.~Balogh.
\newblock A remark on the number of edge colorings of graphs.
\newblock {\em European J. Combin.}, 27(4):565--573, 2006.

\bibitem{balogh2018typical}
J.~Balogh and L.~Li.
\newblock The typical structure of {G}allai colorings and their extremal
  graphs.
\newblock {\em SIAM J. Discrete Math.}, 33(4):2416--2443, 2019.

\bibitem{balogh2017container}
J.~Balogh and J.~Solymosi.
\newblock On the number of points in general position in the plane.
\newblock {\em Discrete Anal.}, pages No. 16, 20, 2018.

\bibitem{balogh2016further}
J.~Balogh and A.~Z. Wagner.
\newblock Further applications of the container method.
\newblock In {\em Recent trends in combinatorics}, volume 159 of {\em IMA Vol.
  Math. Appl.}, pages 191--213. Springer, [Cham], 2016.

\bibitem{BHS}
F.~S. Benevides, C.~Hoppen, and R.~M. Sampaio.
\newblock Edge-colorings of graphs avoiding complete graphs with a prescribed
  coloring.
\newblock {\em Discrete Math.}, 340(9):2143--2160, 2017.

\bibitem{clemens2018colourings}
D.~Clemens, S.~Das, and T.~Tran.
\newblock Colourings without monochromatic disjoint pairs.
\newblock {\em European J. Combin.}, 70:99--124, 2018.

\bibitem{das2019colouring}
S.~Das, R.~Glebov, B.~Sudakov, and T.~Tran.
\newblock Colouring set families without monochromatic {$k$}-chains.
\newblock {\em J. Combin. Theory Ser. A}, 168:84--119, 2019.

\bibitem{Deshouillers1999sum-free}
J.-M. Deshouillers, G.~A. Freiman, V.~S{\'{o}}s, and M.~Temkin.
\newblock On the structure of sum-free sets. {II}.
\newblock {\em Ast{\'{e}}risque}, (258):149--161, 1999.
\newblock Structure theory of set addition.

\bibitem{diananda1969maximal}
P.~H. Diananda and H.~P. Yap.
\newblock Maximal sum-free sets of elements of finite groups.
\newblock {\em Proc. Japan Acad.}, 45:1--5, 1969.

\bibitem{E}
P.~Erd{\H{o}}s.
\newblock Some new applications of probability methods to combinatorial
  analysis and graph theory.
\newblock {\em Proceedings of the {F}ifth {S}outheastern {C}onference on
  {C}ombinatorics, {G}raph {T}heory and {C}omputing ({F}lorida {A}tlantic
  {U}niv., {B}oca {R}aton, {F}la., 1974)}, pages 39--51, 1974.

\bibitem{FOSU}
V.~Falgas-Ravry, K.~O'Connell, and A.~Uzzell.
\newblock Multicolor containers, extremal entropy, and counting.
\newblock {\em Random Structures Algorithms}, 54(4):676--720, 2019.

\bibitem{green2005re-lemma}
B.~Green.
\newblock A {S}zemer\'{e}di-type regularity lemma in abelian groups, with
  applications.
\newblock {\em Geom. Funct. Anal.}, 15(2):340--376, 2005.

\bibitem{green2005sum}
B.~Green and I.~Z. Ruzsa.
\newblock Sum-free sets in abelian groups.
\newblock {\em Israel J. Math.}, 147:157--188, 2005.

\bibitem{han2018maximum}
H.~H\`an and A.~Jim\'{e}nez.
\newblock Maximum number of sum-free colorings in finite abelian groups.
\newblock {\em Israel J. Math.}, 226(2), 2018.

\bibitem{hoppen2012edge}
C.~Hoppen, Y.~Kohayakawa, and H.~Lefmann.
\newblock Edge colourings of graphs avoiding monochromatic matchings of a given
  size.
\newblock {\em Combin. Probab. Comput.}, 21(1-2):203--218, 2012.

\bibitem{hoppen2012hypergraphs}
C.~Hoppen, Y.~Kohayakawa, and H.~Lefmann.
\newblock Hypergraphs with many {K}neser colorings.
\newblock {\em European J. Combin.}, 33(5):816--843, 2012.

\bibitem{hoppen2014edge}
C.~Hoppen, Y.~Kohayakawa, and H.~Lefmann.
\newblock Edge-colorings of graphs avoiding fixed monochromatic subgraphs with
  linear {T}ur\'{a}n number.
\newblock {\em European J. Combin.}, 35:354--373, 2014.

\bibitem{hoppen2015edge}
C.~Hoppen, Y.~Kohayakawa, and H.~Lefmann.
\newblock Edge-colorings of uniform hypergraphs avoiding monochromatic
  matchings.
\newblock {\em Discrete Math.}, 338(2):262--271, 2015.

\bibitem{hoppen2015color}
C.~Hoppen and H.~Lefmann.
\newblock Edge-colorings avoiding a fixed matching with a prescribed color
  pattern.
\newblock {\em European J. Combin.}, 47:75--94, 2015.

\bibitem{hoppen2016coloring}
C.~Hoppen, H.~Lefmann, and K.~Odermann.
\newblock A coloring problem for intersecting vector spaces.
\newblock {\em Discrete Math.}, 339(12):2941--2954, 2016.

\bibitem{hoppen2017graphs}
C.~Hoppen, H.~Lefmann, and K.~Odermann.
\newblock On graphs with a large number of edge-colorings avoiding a rainbow
  triangle.
\newblock {\em European J. Combin.}, 66:168--190, 2017.

\bibitem{hoppen2017rainbow}
C.~Hoppen, H.~Lefmann, and K.~Odermann.
\newblock A rainbow {E}rd{\H{o}}s-{R}othschild problem.
\newblock {\em SIAM J. Discrete Math.}, 31(4):2647--2674, 2017.

\bibitem{sum-counting}
S.~Huczynska.
\newblock Beyond sum-free sets in the natural numbers.
\newblock {\em Electron. J. Combin.}, 21(1):Paper No. 1.21, 20, 2014.

\bibitem{KSV2009}
D.~Kr\'{a}l, O.~Serra, and L.~Vena.
\newblock A combinatorial proof of the removal lemma for groups.
\newblock {\em J. Combin. Theory Ser. A}, 116(4):971--978, 2009.

\bibitem{lefmann2013exact}
H.~Lefmann and Y.~Person.
\newblock Exact results on the number of restricted edge colorings for some
  families of linear hypergraphs.
\newblock {\em J. Graph Theory}, 73(1):1--31, 2013.

\bibitem{lefmann2009colourings}
H.~Lefmann, Y.~Person, V.~R\"{o}dl, and M.~Schacht.
\newblock On colourings of hypergraphs without monochromatic {F}ano planes.
\newblock {\em Combin. Probab. Comput.}, 18(5):803--818, 2009.

\bibitem{lefmann2010structural}
H.~Lefmann, Y.~Person, and M.~Schacht.
\newblock A structural result for hypergraphs with many restricted edge
  colorings.
\newblock {\em J. Comb.}, 1(3-4):441--475, 2010.

\bibitem{2022Integer}
X.~Li, H.~Broersma, and L.~Wang.
\newblock Integer colorings with no rainbow 3-term arithmetic progression.
\newblock {\em Electron. J. Combin.}, 29(2):Paper No. 2.28, 2022.

\bibitem{lin2022integer}
H.~Lin, G.~Wang, and W.~Zhou.
\newblock Integer colorings with no rainbow $k$-term arithmetic progression.
\newblock {\em European J. Combin.}, 104:103547, 2022.

\bibitem{liu2017maximum}
H.~Liu, M.~Sharifzadeh, and K.~Staden.
\newblock On the maximum number of integer colourings with forbidden
  monochromatic sums.
\newblock {\em Electron. J. Combin.}, 28(1):Paper No. 1.59, 35, 2021.

\bibitem{pikhurko2012maximum}
O.~Pikhurko and Z.~B. Yilma.
\newblock The maximum number of {$K_3$}-free and {$K_4$}-free edge 4-colorings.
\newblock {\em J. Lond. Math. Soc. (2)}, 85(3):593--615, 2012.

\bibitem{ST}
D.~Saxton and A.~Thomason.
\newblock Hypergraph containers.
\newblock {\em Invent. Math.}, 201(3):925--992, 2015.

\bibitem{additive}
T.~Tao and V.~Vu.
\newblock {\em Additive combinatorics}, volume 105 of {\em Cambridge Studies in
  Advanced Mathematics}.
\newblock Cambridge University Press, Cambridge, 2010.

\bibitem{yuster1996number}
R.~Yuster.
\newblock The number of edge colorings with no monochromatic triangle.
\newblock {\em J. Graph Theory}, 21(4):441--452, 1996.

\end{thebibliography}
\end{document}